\documentclass[12pt]{amsart}
\usepackage{lineno}

\usepackage{amssymb}
\usepackage{amsmath}

\usepackage{color}
\newtheorem{thm}{Theorem}

\newtheorem{prop}[thm]{Proposition}

\newtheorem{lem}[thm]{Lemma}

\newenvironment{rema}[1]{\noindent {\em Remark.} #1}{}

\begin{document}
\title{On the rank of $\pi_1(\textup{Ham}) $}
\author{Andr\'es Pedroza}
\address{Facultad de Ciencias\\
           Universidad de Colima\\
           Bernal D\'{\i}az del Castillo No. 340\\
           Colima, Col., Mexico 28045}
\email{andres\_pedroza@ucol.mx}

\begin{abstract}
We show that for any positive integer $k$
there exists a closed symplectic $4$-manifold, such that the rank
of the fundamental group of the group of Hamiltonian diffeomorphisms
is at least $k.$ 
\end{abstract}


\maketitle

\section{Introduction}

The problem of determining the homotopy type of the group of Hamiltonian
diffeomorphisms for a closed symplectic manifold is  nowadays a far reaching problem 
in symplectic topology.
In order to have an idea of the limited knowledge in the subject, absolutely nothing is known 
about the homotopy type of the group $\textup{Ham}(\mathbb{T}^4,\omega)$ where
$(\mathbb{T}^4,\omega)$ is the $4$-dimensional torus with the standard symplectic form.
Note that we stated an example of a 4-dimensional manifold, since for 2-dimensional symplectic
manifolds the problem is partially understood due to the fact that in two dimensions
symplectic geometry  agrees with area and orientation preserving geometry. 
See for instance \cite[Sec. 7.2]{polterovich-thegeometry}.
In this direction,
$\textup{Ham}(S^2,\omega)=\textup{Symp}_0(S^2,\omega)$ has the same homotopy type
as $SO(3)$ \cite{smale}; and 
for the surface of genus $g\geq 1$,
$\textup{Ham}(M_g,\omega)$
 is simply connected.

In higher dimensions there are  some cases where the homotopy type of $\textup{Ham}(M,\omega)$
is completely understood, due the techniques of holomorphic curves. For instance,
$\textup{Ham}(\mathbb{C}P^2,\omega_{FS})$ has the homotopy type of $PU(3)$;
$\textup{Ham}(\mathbb{C}P^1\times\mathbb{C}P^1,\omega_{FS}\oplus \omega_{FS})$ has the 
homotopy type of $SO(3)\times SO(3)$. These results are due to M. Gromov \cite{gromov-psudo}.
The rational homotopy type
for the case of one-point blow up of $(\mathbb{C}P^2,\omega_\textup{FS})$ was  settled
by M. Abreu and D. McDuff  in \cite{abreu-mcduff-topology}. 
For more examples, see the work of 
F. Lalonde and M. Pinsonnault  \cite{Lalonde-Pin};  and  J. Evans \cite{Evans-symplectic-mapping}.

Leave behind the problem of determining the homotopy type 
and focus on first stage of the problem: the fundamental group
of $\textup{Ham}(M,\omega)$. 
Recall that the fundamental group of a topological group is an abelian group.
Hence is natural to ask: Given any positive integer $k$, does there exists a 
symplectic manifold such that the free part of $\pi_1(\textup{Ham}(M,\omega))$
is isomorphic to  $\mathbb{Z}^k$. Using cartesian products of symplectic manifolds together with
Seidel's representation \cite{seidel-pi1of} or Weinstein's morphism \cite{weinstein-coho} if possible 
to provide a weak answer to the problem.  
Namely, is possible to construct a symplectic manifold $(M,\omega)$ such that the rank
of $\pi_1(\textup{Ham}(M,\omega))$ is at least  $k$. 
See for instance \cite{pedroza-seidel}, where Seidel's morphism on cartesian products is studied.
In this note we arrive to the same
conclusion but on 4-dimensional symplectic manifolds. 
That is, in the smallest possible 
dimension. 

\begin{thm}
\label{t:main}
Given a positive integer $k$, there exists a closed,  connected and simply connected symplectic $4$-manifold
$(M,\omega)$
such that
\begin{eqnarray*}
\textup{rank } \pi_1(\textup{Ham}(M,\omega)) \geq k.
\end{eqnarray*}
\end{thm}

The proof that we provide is a hands-on proof. 
The symplectic $4$-manifold of the theorem turns out to be
the blow up of $(\mathbb{C}P^2,\omega_\textup{FS})$ at $k$ 
points of  distinct weights.
The techniques used throughout this note are soft techniques of symplectic topology, where
Weinstein's morphism plays a key role.

If $\textup{Symp}_0(M,\omega)$ stands for
the connected component of $\textup{Symp}(M,\omega)$
that contains the identity map, then  the inclusion
$\textup{Ham}(M,\omega)\subset\textup{Symp}_0(M,\omega)$
induces an injective map
$$
\pi_1(\textup{Ham}(M,\omega))\to  \pi_1(\textup{Symp}_0(M,\omega))
$$
 due to  the Flux morphism, \cite[Ch. 10]{ms}. Therefore, Theorem  \ref{t:main} also holds
if  the group $\textup{Ham}(M,\omega)$ is replaced by $\textup{Symp}_0(M,\omega)$.

Unlike the group of Hamiltonian diffeomorphisms of a closed symplectic manifold, the
group of symplectic diffeomorphisms is not necessarily connected.
Therefore, following the same line of ideas is  natural to ask
if given a positive number $k$ does there exists a closed connected symplectic manifold $(M,\omega)$ such that 
the number of connected components of $\textup{Symp}(M,\omega)$ is equal to $k$. Recently,
D. Aroux and I. Smith solved this problem in \cite[Thm 1.3]{auroux-smith} via Floer-theoretic arguments.

As a byproduct of the arguments used to prove the main result, we are also able to
show that  Calabi's morphism on the one-point blow up of $(\mathbb{R}^4,\omega_0)$
is non trivial.  
The first examples of  open manifolds whose Calabi's morphism is non trivial 
are due to   A. Kislev \cite{kislev-compact}. 
Alongside we prove that
 the rank
of the fundamental group of $\textup{Ham}(\widetilde M,\widetilde \omega_\kappa)$ is positive 
where $(\widetilde M,\widetilde \omega_\kappa)$ is the 
 one-point blow of weight $\kappa$ of $(M,\omega)$.
Hence our result improves the one obtained by D. McDuff \cite{mcduff-blowup},
since  information about $\pi_1(M)$ and $\pi_2(M)$ is irrelevant in our arguments.

\begin{thm}
\label{t:onepoint}
Let $(\widetilde M,\widetilde \omega_\kappa)$ be the one-point  blow up of  weight $\kappa$
of  the closed manifold $(M,\omega)$. 
Then for infinitely many values of $\kappa$,
 the rank of $\pi_1(\textup{Ham}(\widetilde M,\widetilde \omega_\kappa))$ is positive.
\end{thm}

It is worth mentioning the case $M=\mathbb{T}^4$ with the standard symplectic form. Therefore,
$\pi_1(\textup{Ham}(\widetilde{\mathbb{T}}^4,\widetilde\omega_\kappa))$  has positive rank. However,
nothing is known about the group $\pi_1(\textup{Ham}({\mathbb{T}}^4,\omega))$.

We are grateful to L. Polterovich from bringing \cite{kislev-compact} to our attention.

\section{Preliminary computations}
\label{s:compu}

\subsection{A compactly supported path of Hamiltonian diffeomorphisms on $(\mathbb{R}^4,\omega_0)$}
Consider $1$-periodic smooth functions $a_1,a_2,a_3,a_4 :\mathbb{R}\to \mathbb{R}$ such
that $a_1(0)=a_4(0)=1$ and $a_2(0)=a_3(0)=0$. Let $\alpha :\mathbb{R}\to \mathbb{R}$ 
be also a smooth function such that $\alpha(0)\in \mathbb{Z}$.
Furthermore, the functions $a_j$ are 
also subject to the condition that
\begin{eqnarray}
\label{e:mat}
A_t:= 
\begin{pmatrix}
a_1(t)e^{2\pi i\, \alpha(t)} & a_2(t) e^{-2\pi i t} \\
a_3(t) e^{2\pi i\, \alpha(t)} &a_4(t) e^{-2\pi i t} 
 \end{pmatrix} \ \  \ \  t\in[0,1] 
\end{eqnarray}
is a $2\times 2$ unitary matrix. Notice that the only constraint on $\alpha$
is $\alpha(0)\in \mathbb{Z}$, it does not have to be periodic unlike the functions $a_j$.
Thus $\{A_t\}_{0\leq t \leq 1}$ is a path in $U(2)$ that
stars at the identity.
Let $\psi^{\bf a}_t:(\mathbb{C}^2,\omega_0)\to(\mathbb{C}^2,\omega_0)$ be the 
path of Hamiltonian diffeomorphism induced by $\{A_t\}$.  A direct computation
yields the Hamiltonian $H^{\bf a}_t $ and time-dependent vector field 
$X^{\bf a}_t $ induced by $\{\psi^{\bf a}_t\}$.

\begin{lem}
\label{l:description}
The path $\{\psi^{\bf a}_t\}_{0\leq t\leq 1}$ induced by the path of unitary matrices $\{A_t\}$ 
induces in
$(\mathbb{R}^4,\omega_0)$ the time-dependent vector field
\begin{eqnarray*}
X^{\bf a}_t 
&=&   
\left \{2\pi y_1(a_2^2-a_1^2\alpha^\prime)   +x_2(a_1^\prime a_3+a_2^\prime a_4)+2\pi y_2 (a_2a_4-a_1a_3\alpha^\prime)   \right\}\frac{\partial}{\partial x_1}  \\
& &   
+\left \{2\pi x_1(a_1^2\alpha^\prime-a_2^2)   +2\pi x_2(a_1a_3\alpha^\prime-a_2a_4)+y_2 (a_1^\prime a_3+a_2^\prime a_4)   \right\}\frac{\partial}{\partial y_1} \\
& &   
+\left \{x_1 (a_1a_3^\prime +a_2a_4^\prime) +2\pi y_1(a_2 a_4-a_1 a_3\alpha^\prime)+2\pi y_2 (a_4^2-a_3^2\alpha^\prime)   \right\}\frac{\partial}{\partial x_2} \\
& &   
+\left \{2\pi x_1(a_1 a_3\alpha^\prime-a_2 a_4)   -y_1(a_1a_3^\prime +a_2 a_4^\prime )+2\pi x_2 (a_3^2\alpha^\prime-a_4^2)   \right\}\frac{\partial}{\partial y_2} ,
\end{eqnarray*}
and Hamiltonian function
\begin{eqnarray*}
H^{\bf a}_t(x_1,y_1,x_2,y_2) 
&=&   \pi (-a_1^2\alpha^\prime+a_2^2) (x_1^2+y_1^2) + \pi (-a_3^2\alpha^\prime+a_4^2) (x_2^2+y_2^2) \\
& & +  2 \pi (a_2 a_4-a_1a_3\alpha^\prime) (x_1x_2+y_1y_2)  \\
& & +(a_1a_3^\prime+a_2 a_4^\prime) (x_1y_2-x_2y_1).
\end{eqnarray*}
\end{lem}

For $r>0$ denote by $B_r$ the open ball of radius $r$ in $\mathbb{R}^4$ centered at the origin.
Using the fact that the functions $a_j$ define the
unitary matrix  (\ref{e:mat}),  it follows that the integral of $H^{\bf a}_t$ over
the ball $B_r$ only depends on $\alpha$ and $r$.
\begin{lem}
\label{l:inte}
If $H^{\bf a}_t$ is the Hamiltonian function of Lemma \ref{l:description},
then
\begin{eqnarray*}
\int_0^1 \int_{B_r} H^{\bf a}_t\, \omega_0^2\, dt =\left( \frac{\pi^3 r^6}{6} \right) (1+\alpha(0)-\alpha(1)).
\end{eqnarray*}
\end{lem}
\begin{proof}
The integral over $B_r$ of the last two terms of $H^{\bf a}_t$ are zero. 
Now since the functions $a_j$ are the entries of a unitary matrix, it follows that
\begin{eqnarray*}
\int_0^1 \int_{B_r} H^{\bf a}_t\, \omega_0^2\, dt   &=& 
\pi \int_{B_r}  (x_1^2+y_1^2) \, \omega_0^2 \cdot \int_0^1  -a_1^2\alpha^\prime+a_2^2 -a_3^2\alpha^\prime+a_4^2\, dt \\
 &=& \pi\left( \frac{\pi^2 r^6}{6} \right) \int_0^1  1-\alpha^\prime\, dt \\
 &=& \pi\left( \frac{\pi^2 r^6}{6} \right) (1+\alpha(0)-\alpha(1)).
\end{eqnarray*}
\end{proof}

We are interested in the case when the above integral is non zero. 
Thereby, by the last result we neglect the functions $a_j$  and focus our attention on $\alpha$.

Next we used a bump function to obtain a compactly supported Hamiltonian path.
This is the reason that  we considered a Hamiltonian path instead of a Hamiltonian loop.
If a Hamiltonian function induces a loop, then multiplying it by a bump function 
will no longer induce a loop.
To that end
fix $r_0,R_0\in \mathbb{R}$ such that $0<r_0<R_0$ and also fix a  bump function
$\rho: \mathbb{R}^4\to \mathbb{R}$ such that $\rho\equiv 1$ on $B_{r_0}$
and $\rho\equiv 0$ on $\mathbb{R}^4\setminus B_{R_0}$.
Consider  the compactly supported smooth function  
$H^{\bf a,\rho}_t: (\mathbb{R}^4,\omega_0)\to\mathbb{R}$  defined as
$$
H^{\bf a,\rho}_t:= \rho\cdot H^{\bf a}_t
$$
and let $\{\psi^{\bf a,\rho}_t\}_{0\leq t\leq 1}$ be the induced Hamiltonian path.
Observe that  $\psi^{\bf a,\rho}_t \in \textup{Ham}^c(\mathbb{R}^4,\omega_0)$
and that on the ball  $B_{r_0}$ we have that $\psi^{\bf a,\rho}_t=\psi^{\bf a}_t$ 
and $H^{\bf a,\rho}_t=H^{\bf a}_t$ for all $t\in[0,1]$.

\subsection{A compactly supported loop of Hamiltonian diffeomorphisms on $(\mathbb{R}^4,\omega_0)$}
Next we define a loop  in $\textup{Ham}^c(\mathbb{R}^4,\omega_0)$ based at the identity
by concatenating two
paths of the kind defined above. Fix $a_1,\ldots a_4,b_1,\ldots b_4 $ smooth 1-periodic functions
and $\alpha$ and $\beta $ also smooth functions, such that all functions are  subject to the 
conditions previously imposed.
Further, we also impose the condition that the two  path of matrices $\{A_t\}$ and $\{B_t\}$ agree 
in a neighborhood of $t=1$. 
Therefore, the exponent functions satisfy
$\alpha(1)-\beta(1)\in \mathbb{Z}$.

Define the Hamiltonian  loop $\psi=\{\psi_t \}_{0\leq t\leq 2}$ in 
$\textup{Ham}^c(\mathbb{R}^4,\omega_0)$  as
\begin{eqnarray}
\label{e:loop}
\psi_t :=
\left\{
	\begin{array}{ll}
		\psi^{\bf a,\rho}_t  & t\in[0,1] \\
		\psi^{\bf b,\rho}_{1-t}  & t\in[1,2].
	\end{array}
\right.
\end{eqnarray}
Thus, its compactly supported Hamiltonian function is given by
\begin{eqnarray}
\label{e:Hamloop}
 H_t :=
\left\{
	\begin{array}{ll}
		H^{\bf a,\rho}_t  & t\in[0,1] \\
		H^{\bf b,\rho}_{1-t}  & t\in[1,2].
	\end{array}
\right.
\end{eqnarray}

The above observations on the paths
$\{\psi^{\bf a,\rho}_t\}$ and $\{\psi^{\bf b,\rho}_t\}$
and the computation of Lemma \ref{l:inte} 
imply the  following facts about the Hamiltonian loop  $\psi.$

\begin{prop}
\label{p:integral} 
Given $r_0,R_0\in \mathbb{R}$ such that $R_0>r_0>0$ there is a
Hamiltonian loop $\psi=\{\psi_t \}_{0\leq t\leq 2}$ defined as in (\ref{e:loop}) 
such that it is supported in $B_{R_0}$ and on $B_{r_0}$ it agrees with a loop
of unitary matrices. Moreover, its Hamiltonian function $H_t$ satisfies
\begin{eqnarray}
\label{e:winding}
\int_0^2 \int_{B_{r_0}} H_t\, \omega_0^2\, dt =\left( \frac{\pi^3 {r_0}^6}{6} \right) (\alpha(0)-\alpha(1)-\beta(0)+\beta(1)).
\end{eqnarray}
\end{prop}

The condition that the paths $\{\psi^{\bf a,\rho}_t\}$ and $\{\psi^{\bf b,\rho}_t\}$
agree in a neighborhood of 1, imply that $\alpha(0)-\alpha(1)-\beta(0)+\beta(1)$ is a integer.
Thus Eq. (\ref{e:winding}) is a kind of {\em winding number} for the loop of unitary matrices. 
Furthermore, it is possible to choose the functions $\alpha$ and $\beta$ so that 
the paths agree in a neighborhood of 1 and $\alpha(0)-\alpha(1)-\beta(0)+\beta(1)$ is non zero. 
From now on we fix the Hamiltonian loop $\psi$ in  $\textup{Ham}^c(\mathbb{R}^4,\omega_0)$
defined in (\ref{e:loop}) so that 
$\alpha(0)-\alpha(1)-\beta(0)+\beta(1)$ is equal to 1. 
The actual value is not important, what is relevant at this point  is that it is  a non zero integer.

On an open manifold $(M,\omega)$  the Calabi morphism
$\textup{Cal}: \pi_1(\textup{Ham}^c(M,\omega))\to \mathbb{R}$
is defined as
\begin{eqnarray*}
\textup{Cal}(\phi):=
\int_0^1 \int_{M} F_t\, \omega^n\, dt ,
\end{eqnarray*}
where the loop $\phi$
is generated by the compactly supported Hamiltonian  $F_t$. Moreover, if $(M,\omega)$
is exact then Calabi's morphism is identically zero.
Consequently, for the  the loop $\psi$ defined in  (\ref{e:loop})  we have that
\begin{eqnarray}
\label{e:null}
0=\int_0^2 \int_{\mathbb{R}^4} H_t\, \omega_0^2\, dt ,
\end{eqnarray}
in contrast with the integral over $B_{r_0}$ that is non zero.

Equation  (\ref{e:null}) will be useful when we consider the loop $\psi$
in a Darboux chart on a closed manifold. 
For, in this case the normalization condition corresponds to zero mean.

\section{A loop of Hamiltonian diffeomorphisms  in the one-point blow up of infinite order}

In \cite{pedroza-hamiltonian} it is proved that if a closed symplectic manifold
admits a Hamiltonian circle action, then after blowing up one point the fundamental
group of the group of Hamiltonian diffeomorphisms has positive rank.
In this section we prove that the above result always holds, namely we prove
Theorem \ref{t:onepoint}. Henceforth, the hypothesis about the Hamiltonian circle
action is no longer required. We restrict to the four-dimensional case in view of the main
result of this paper; 
however the result presented in this section holds on any symplectic manifold of dimension
greater than or equal than four.

This section can be considered has the first step of the proof of the main theorem.
Recall that for any $R_0>0$,  the loop $\psi=\{\psi_t\}_{0\leq t\leq 2}$ defined in (\ref{e:loop}) 
is supported in the open ball $B_{R_0}\subset (\mathbb{R}^4,\omega_0)$. 
Let $(M,\omega)$ be a closed rational symplectic 4-manifold. For $R_0>0$ small enough,
by Darboux's theorem the loop $\psi\in \textup{Ham}^c(\mathbb{R}^4,\omega_0)$
can be regarded in $\textup{Ham}(M,\omega)$.

Since we are in a closed symplectic manifold $(M,\omega)$, we must normalized
the Hamiltonian function $H_t:(M,\omega)\to \mathbb{R}$.
Thus, for $t\in[0,1]$ define
$$
c_t:=\frac{1}{\textup{Vol}(M,\omega^2)}
 \int_{M} H_t\,{\omega^2}.
$$
Afterwards define $H^{\bf N}_t:M\to\mathbb{R}$ as
 $H^{\bf N}_t:= H_t-c_t$. Thus, $H^{\bf N}_t$  is the normalized Hamiltonian function
that induces the same loop $\psi$ in $\textup{Ham}(M,\omega)$.

Call $\iota B_{r_0}\subset M$ the image of $B_{r_0}\subset \mathbb{R}^4$
under the Darboux embedding.
Then the loop $\psi=\{\psi_t\}$ satisfies the following:
\begin{itemize}
\item $\psi_t(\iota(0)))=\iota (0)$ for all $t\in [0,2]$
\item $\psi_t$ behaves like a unitary matrix  on $\iota B_{r_0}$ for all $t\in [0,2]$.
\end{itemize}

Using  the embedded ball $\iota B_{r_0}\subset M$, define 
$(\widetilde M,\widetilde\omega_{r_0})$ to
be the one-point blow up  at $\iota(0)$ of $(M,\omega)$ of weight $r_0$. 
From the  above remarks on the loop $\psi$, it follows from  \cite[Sec. 3]{pedroza-hamiltonian} that $\psi$ 
induces a Hamiltonian  loop
$\widetilde\psi=\{\widetilde \psi_t\}_{0\leq t\leq 2}$ in $\textup{Ham}(\widetilde M,\widetilde\omega_{r_0})$.
For  appropriate values of $r_0$, we claim that 
$[\widetilde \psi]$ has infinite order in  $\pi_1(\textup{Ham}(\widetilde M,\widetilde\omega_{r_0})).$ 
We prove this using Weinstein's morphism 
$$
\mathcal{A}: \pi_1(\textup{Ham}(\widetilde M,\widetilde\omega_{r_0}))
\to \mathbb{R}/\mathcal{P}(\widetilde M,\widetilde\omega_{r_0}),
$$
where $\mathcal{P}(\widetilde M,\widetilde\omega_{r_0})$ is the period group.

Next we prove Theorem \ref{t:onepoint} that was stated at the Introduction. 
We give a more precise statement of the theorem in terms 
of the loop of Hamiltonian diffeomorphisms $\psi$ defined above and the weight of 
the blow up. Keep in mind that we state this result for four-dimensional symplectic
manifolds, but the same argument works in higher dimensions.

\begin{thm}
\label{t:onlyone2}
Let $r_0>0$ such that $\pi r_0^2$ is a transcendental number.
If $(M,\omega)$ is a rational symplectic 4-manifold, then the induced loop
$[\widetilde \psi]$ has infinite order in  $\pi_1(\textup{Ham}(\widetilde M,\widetilde\omega_{r_0})).$
\end{thm}
\begin{proof}
Since $(M,\omega)$ is rational, there are $q_1,\ldots, q_s\in \mathbb{Q}$ such that 
$\mathcal{P}(M,\omega)=\mathbb{Z}\langle q_1,\ldots, q_s \rangle$. In fact,
$\mathcal{P}(\widetilde M,\widetilde\omega_{r_0})=\mathbb{Z}\langle q_1,\ldots, q_s, \pi r_0^2 \rangle$
since the area of the line in the exceptional divisor is $\pi r_0^2$.

From \cite[Thm. 1.1]{pedroza-hamiltonian}, it is possible to  compute $\mathcal{A}(\widetilde\psi)$ in terms 
solely of  the loop $\psi$. 
Namely,
\begin{eqnarray*}
\mathcal{A}(\widetilde\psi)=
\left[  \mathcal{A}(\psi) +  \frac{1}{\textup{Vol} (\widetilde M,\widetilde\omega_{r_0}^2)}
\int_0^2 \int_{\iota B_{r_0}} H^{\bf N}_t\, {\omega^2}\,dt
\right].
\end{eqnarray*}
Since $\{\psi_t\}$ is null  homotopic, it follows that $\mathcal{A}(\psi)=0 $ in $\mathbb{R}/
\mathcal{P}(M,\omega)$.
Observe that the integral of $H^{\bf N}_t$ over $\iota B_{r_0}\subset M$ is the same as the integral over
$B_{r_0}\subset \mathbb{R}^4$. Therefore, from Proposition \ref{p:integral} we have that  
\begin{eqnarray*}
\int_0^2 \int_{\iota B_{r_0}} H^{\bf N}_t\, {\omega^2} \,dt
&=&
\int_0^2 \int_{\iota B_{r_0}} H_t-c_t\, {\omega^2} \,dt\\
&=&
\left( \frac{\pi^3 {r_0}^6}{6} \right) \cdot 1 - \textup{Vol}  (B_{r_0},\omega^2)\int_0^2  c_t\,dt.
\end{eqnarray*}

The functions $H_t$ are supported $B_{R_0}\subset\mathbb{R}^4$, thus they can also
be considered as a functions on $M$.
Hence using  Eq.  (\ref{e:null}) it follows that 
\begin{eqnarray*}
\int_0^2  c_t\,dt 
&=&
\int_0^2 \frac{1}{\textup{Vol}(M,\omega^2)}
 \int_{M} H_t\,{\omega^2}\, dt\\
&=&
 \frac{1}{\textup{Vol}(M,\omega^2)}
\int_0^2 \int_{M} H_t\,{\omega^2}\,dt=0.
 \end{eqnarray*}

If $V$ stands for $\textup{Vol}(M,\omega^2)$,
then $\textup{Vol} (\widetilde{M},\widetilde\omega_{r_0}^2)=V-\pi^2r_0^4/2$.
Substituting the above computations, we have
\begin{eqnarray*}
\mathcal{A}(\widetilde\psi)=
\left[  \frac{1}{V-\pi^2r_0^4/2}\left( \frac{\pi^3 {r_0}^6}{6} \right) 
\right] \in \mathbb{R}/\mathbb{Z}\langle q_1,\ldots ,q_s, \pi r_0^2 \rangle.
\end{eqnarray*}
Then the equation $\mathcal{A}(\widetilde\psi^m)=0$, for $m\in \mathbb{N}$,
is equivalent to a polynomial equation on $\pi r_0^2$ with rational coefficients. Recall that
$V\in\mathbb{Q}.$
Since $\pi r_0^2$ is transcendental,  the Hamiltonian loop
$\widetilde \psi=[\{\widetilde\psi_t\}_{0\leq t\leq 2}]$ 
has infinite order in  $\pi_1(\textup{Ham}(\widetilde{M},\widetilde\omega_{r_0})).$
\end{proof}

\section{Proof of the main theorem}

The proof of Theorem \ref{t:onlyone2} gives the blueprint that we follow in order to
prove the main result. The proof of the main theorem deals with
$k$ distinct Hamiltonian loops supported in $k$ mutually disjoint balls. 
Thereby, 
the hypothesis in Theorem \ref{t:onlyone2}, about   $\pi r_0$ being 
transcendental, is replaced 
by the following lemma.

\begin{lem} 
\label{l:help}
Given $k\in \mathbb{N}$ there exist $k$ distinct real numbers $y_1,\ldots, y_k$
such that for any  $j,s\in \{1,\ldots,k\}$ the equation
$$
(a+q_1 y_1+\cdots+ q_k y_k) ( b-(y_1^2+\cdots +y_k^2))+cy_j^3+ y_s^3=0
$$
has no solution for any $q_1,\ldots, q_k,a,b,c\in \mathbb{Q}$.
\end{lem}

The proof of the lemma is a consequence of the fact that $\mathbb{R}$ is  an infinite dimensional
$\mathbb{Q}$-vector space.
Next we provide the proof of the main theorem builded on the ideas of the 
proof of Theorem \ref{t:onlyone2}.

For the proof of the main theorem is important to
consider the complex 2-dimensional projective space
$(\mathbb{C}P^2,\omega_{FS})$ endowed 
with its standard Fubini-Study symplectic form.
If $\omega_{FS}$ is normalized so that 
 $(\mathbb{C}P^2\setminus\mathbb{C}P^1 ,\omega_{FS})$
is symplectomorphic to the open ball of radius $R$ in $(\mathbb{R}^4,\omega_0)$, then
the area of a complex line is 
$\pi R^2$. Therefore, the period group of $(\mathbb{C}P^2,\omega_{FS})$
is $\mathbb{Z}\langle \pi R^2\rangle$ and $\textup{Vol}(\mathbb{C}P^2,\omega_{FS}^2)=(\pi R^2)^2/2$.

\begin{proof}[Proof of Theorem \ref{t:main}]
Given  $k>0$, it follows from Lemma \ref{l:help} that 
there are $k$ distinct  numbers $r_1,\ldots, r_k\in \mathbb{R}_{>0}$ where
$y_j:=\pi r_j^2$ for $j\in\{1,\ldots , n\}$.

Let $R_1,\ldots,R_k \in \mathbb{R}$ be 
any numbers 
such that $R_j>r_j$ for all $j\in\{1,\ldots ,k\}$.
On $\mathbb{R}^4$, fix $k$  mutually disjoint balls 
$B_1,\ldots,  B_k$ centered at $p_1,\ldots,p_k$ and
of radii $R_1,\ldots,R_k$ respectively. Inside
each ball $B_j$ fix a a smaller ball  $B_{r_j}$ of radius $r_j$
centered at $p_j$.
Next, consider the complex  2-dimensional projective space
$(\mathbb{C}P^2,\omega_{FS})$ such that the symplectic form
is normalized so that
 $(\mathbb{C}P^2\setminus\mathbb{C}P^1 ,\omega_{FS})$
is symplectomorphic to an open ball of radius $R$  in $(\mathbb{R}^4,\omega_0)$ that contains
all the balls $B_1,\ldots, B_k$ and $\pi R^2$ is a rational number.
Therefore the period group of $(\mathbb{C}P^2,\omega_{FS})$
is $\mathbb{Z}\langle \pi R^2 \rangle$.

Call $\iota_j B_{r_j}\subset \mathbb{C}P^2$
the image of the fixed ball $B_{r_j}\subset\mathbb{R}^4$. Hence in 
$(\mathbb{C}P^2,\omega_{FS})$
 there are $k$ mutually disjoint embedded balls $\iota_1 B_{r_1},\ldots, \iota_k B_{r_k}$.
Denote by $(\mathbb{C}P^2 \#_k\overline{\mathbb{C}P}^2,\widetilde\omega_{\bf r})$ 
the symplectic manifold that is obtained by blowing up the $k$ points
$\iota_1(p_1),\ldots ,\iota_k(p_k)$ in $(\mathbb{C}P^2,\omega_{FS})$
where the weight at $\iota_j(p_j)$ is $r_j.$ That is, the embedded ball 
$\iota_j B_{r_j}$ is removed from the torus for $j\in\{1,\ldots,k\}.$ 
$(\mathbb{C}P^2 \#_k\overline{\mathbb{C}P}^2,\widetilde\omega_{\bf r})$ is the
desired closed symplectic 4-manifold of the main theorem. Notice that it is simply connected.

Next we define  $k$ Hamiltonian loops.
From Proposition \ref{p:integral} for each $j\in\{1,\ldots, k\}$  there is a loop
 $\psi^{(j)}=\{\psi^{(j)}_t\}_{0\leq t\leq 2}$ in $\textup{Ham}^c(\mathbb{R}^4,\omega_0)$
supported in  the ball of radius $R_j$ such that inside the ball of radius $r_j$ it agrees with
a loop of unitary matrices. Therefore, the $k$ loops in $\textup{Ham}^c(\mathbb{R}^4,\omega_0)$ induced
 $k$ loops $\textup{Ham}(\mathbb{C}P^2,\omega_\textup{FS})$ that we also denoted by
$\psi^{(1)},\ldots ,\psi^{(k)}$.

For every  $j\in \{1,\ldots ,k\}$, the loop $\psi^{(j)}$ on $(\mathbb{C}P^2,\omega_\textup{FS})$ 
 behaves as a loop of unitary matrices on each of the embedded ball $\iota_s B_{r_s}$. Moreover, 
if $s\neq j$ it  is the constant loop on $\iota_s B_{r_s}$. 
In any case, the loop $\psi^{(j)}$
fixes the  points  $\iota_1(p_1),\ldots, \iota_k(p_k)$.
Therefore, from 
\cite[Sec. 3]{pedroza-hamiltonian} it follows that $\psi^{(j)}$ induces a loop 
$\widetilde \psi^{(j)}$ in $\textup{Ham}(\mathbb{C}P^2 \#_k\overline{\mathbb{C}P}^2,\widetilde\omega_{\bf r})$.
We claim that the loops $\widetilde \psi^{(1)},\ldots, \widetilde \psi^{(k)}$  generate
a subgroup isomorphic to $\mathbb{Z}^k$ in 
$\pi_1( \textup{Ham} (\mathbb{C}P^2 \#_k\overline{\mathbb{C}P}^2,\widetilde\omega_{\bf r})  )$.

Fix $j\in\{1,\ldots k\}$. The corresponding Hamiltonian $H_t^{(j)}$ of $\psi^{(j)}$
is compactly supported, as before let $H_t^{(j),{\bf N}}: (\mathbb{C}P^2,\omega_\textup{FS})\to\mathbb{R}$ 
be its normalization.
As in the proof of Theorem \ref{t:onlyone2} we have that
\begin{eqnarray*}
\int_0^2 \int_{\iota_j B_{r_j}} H_t^{(j),{\bf N}}\, {\omega^2} \,dt
=\left( \frac{\pi^3 {r_j}^6}{6} \right) \cdot 1
\end{eqnarray*}
and
\begin{eqnarray*}
\mathcal{A}(\widetilde\psi^{(j)})=
\left[  \frac{1}{V-\pi^2(r_1^4+\cdots +r_k^4)/2}\left( \frac{\pi^3 {r_j}^6}{6} \right) 
\right] \in \mathbb{R}/\mathbb{Z}\langle \pi R^2 ,\pi r_1^2 ,\ldots , \pi r_k^2 \rangle.
\end{eqnarray*}
where $V:=\textup{Vol}({\mathbb{C}P}^2,\omega_\textup{FS}^2)\in\mathbb{Q}$.

Since $\pi R^2 \in \mathbb{Q}$
and
 the  numbers $\pi r_1^2 ,\ldots , \pi r_k^2$ were chosen according to 
Lemma \ref{l:help}, it follows that for any  $m\in\mathbb{Z}_{>0}$
the equation $\mathcal{A}((\widetilde\psi^{(j)})^m)=0$ does not hold.
Therefore $[\widetilde\psi^{(j)}]$ has infinite order in 
$\pi_1( \textup{Ham} (\mathbb{C}P^2 \#_k\overline{\mathbb{C}P}^2,\widetilde\omega_{\bf r})   )$.
The same reasoning implies that $\mathcal{A}((\widetilde\psi^{(j)})^m)\neq\mathcal{A}((\widetilde\psi^{(s)})^n)$
for distinct $j,s\in\{1,\ldots,k\}$ and any $m,n\in\mathbb{Z}_{>0}$.
Henceforth the rank of $\pi_1( \textup{Ham} (\mathbb{C}P^2 \#_k\overline{\mathbb{C}P}^2,\widetilde\omega_{\bf r}) )$
is at least $k$.
\end{proof}

\section{Calabi's morphism}


As noted before,  Calabi's morphism on 
$\pi_1(\textup{Ham}^c(\mathbb{R}^4,\omega_0))$ is trivial.
However, the arguments used before show that  Calabi's morphism on 
the one-point blow up of $(\mathbb{R}^4,\omega_0)$ is non trivial

Fix $r>0$, let $(\widetilde{\mathbb{R}^4},\widetilde\omega_r)$ be the blowup
of the origin in $(\mathbb{R}^4,\omega_0)$ of weight $r$. 
As we have seen the loop
$\psi$ in $\textup{Ham}^c(\mathbb{R}^4,\omega_0)$ induces the loop
$\widetilde \psi$ in $\textup{Ham}^c(\widetilde{\mathbb{R}^4},\widetilde\omega_r)$.

\begin{prop}
\label{p:cal}
For any $r>0$, the Calabi morphism on  $\textup{Ham}^c(\widetilde{\mathbb{R}^4},
\widetilde\omega_r)$ does not vanish.
\end{prop}
\begin{proof}
Let $\psi$ be the Hamiltonian loop in $\textup{Ham}^c({\mathbb{R}^4},
\omega_0)$ defined in Eq. \ref{e:loop} and let 
 $\widetilde\psi$ be the induced loop in $\textup{Ham}^c(\widetilde{\mathbb{R}^4},
 \widetilde\omega_r)$. According to \cite[Sec. 1]{pedroza-hamiltonian},
$\textup{Cal}(\widetilde\psi)$ and $\textup{Cal}(\psi)$ are related as
\begin{eqnarray}
\label{e:calabi}
\textup{Cal}(\widetilde\psi) =\textup{Cal}(\psi)
-\frac{1}{2} \int_0^2\int_{B_r} H_t\, \omega^2\, dt 
\end{eqnarray}
where $H_t:(\mathbb{R}^4,\omega)\to\mathbb{R}$ is the compactly supported Hamiltonian
function of the loop $\psi$ defined in Eq. (\ref{e:Hamloop}). Thus from Proposition \ref{p:integral},
\begin{eqnarray*}
\textup{Cal}(\widetilde\psi)&=&
-\frac{\pi^3 {r}^6}{12}  \cdot 1.
\end{eqnarray*}
\end{proof}
\begin{rema}
The above result is true for any dimension; Calabi's morphism on  
$\textup{Ham}^c(\widetilde{\mathbb{R}^{2n}}, \widetilde\omega_r)$ does not vanish for $n\geq 2$ and
 $r>0$. 
 \end{rema}

Finally, note that since Calabi's morphism does not vanish  on
$\pi_1(\textup{Ham}^c(\widetilde{\mathbb{R}^4},\widetilde\omega_r))$
it does not descend to $\textup{Ham}^c(\widetilde{\mathbb{R}^4},\widetilde\omega_r)$.
Furthermore, for the Hamiltonian loop $\widetilde\psi$ in $\textup{Ham}^c(\widetilde{\mathbb{R}^4},
\widetilde\omega_r)$ that appears in the previous proof we have that
\begin{eqnarray*}
\ell (\widetilde\psi) \geq \frac{\pi^3 {r}^6}{12}.
\end{eqnarray*}
Here $\ell (\cdot)$ stands for the Hofer length of the class in $\pi_1(\textup{Ham}(\cdot))$.
For the precise definition of   $\ell (\cdot)$, see
\cite[Sec. 7.3]{polterovich-thegeometry}.


\bibliographystyle{acm}
\bibliography{/Users/Andrés/Dropbox/Documentostex/Ref}    

\begin{thebibliography}{10}

\bibitem{abreu-mcduff-topology}
{\sc Abreu, M., and McDuff, D.}
\newblock Topology of symplectomorphism groups of rational ruled surfaces.
\newblock {\em J. Amer. Math. Soc. 13}, 4 (2000), 971--1009 (electronic).

\bibitem{auroux-smith}
{\sc Auroux, D., and Smith, I.}
\newblock Fukaya categories of surfaces, spherical objects, and mapping class
  groups, \texttt{arXiv.2006.09689}.

\bibitem{Evans-symplectic-mapping}
{\sc Evans, J.~D.}
\newblock Symplectic mapping class groups of some {S}tein and rational
  surfaces.
\newblock {\em J. Symplectic Geom. 9}, 1 (2011), 45--82.

\bibitem{gromov-psudo}
{\sc Gromov, M.}
\newblock Pseudoholomorphic curves in symplectic manifolds.
\newblock {\em Invent. Math. 82}, 2 (1985), 307--347.

\bibitem{kislev-compact}
{\sc Kislev, A.}
\newblock Compactly supported {H}amiltonian loops with a non-zero {C}alabi
  invariant.
\newblock {\em Electron. Res. Announc. Math. Sci. 21\/} (2014), 80--88.

\bibitem{Lalonde-Pin}
{\sc Lalonde, F., and Pinsonnault, M.}
\newblock The topology of the space of symplectic balls in rational
  4-manifolds.
\newblock {\em Duke Math. J. 122}, 2 (2004), 347--397.

\bibitem{mcduff-blowup}
{\sc McDuff, D.}
\newblock The symplectomorphism group of a blow up.
\newblock {\em Geom. Dedicata 132\/} (2008), 1--29.

\bibitem{ms}
{\sc McDuff, D., and Salamon, D.}
\newblock {\em Introduction to symplectic topology}, second~ed.
\newblock Oxford Mathematical Monographs. The Clarendon Press, Oxford
  University Press, New York, 1998.

\bibitem{pedroza-seidel}
{\sc Pedroza, A.}
\newblock Seidel's representation on the {H}amiltonian group of a {C}artesian
  product.
\newblock {\em Int. Math. Res. Not. IMRN}, 14 (2008), Art. ID rnn049, 19.

\bibitem{pedroza-hamiltonian}
{\sc Pedroza, A.}
\newblock {H}amiltonian loops on the symplectic blow up.
\newblock {\em J. Symplectic Geom. 16}, 3 (2018), 839--856.

\bibitem{polterovich-thegeometry}
{\sc Polterovich, L.}
\newblock {\em The geometry of the group of symplectic diffeomorphisms}.
\newblock Lectures in Mathematics ETH Z\"urich. Birkh\"auser Verlag, Basel,
  2001.

\bibitem{seidel-pi1of}
{\sc Seidel, P.}
\newblock {$\pi_1$} of symplectic automorphism groups and invertibles in
  quantum homology rings.
\newblock {\em Geom. Funct. Anal. 7}, 6 (1997), 1046--1095.

\bibitem{smale}
{\sc Smale, S.}
\newblock Diffeomorphisms of the {$2$}-sphere.
\newblock {\em Proc. Amer. Math. Soc. 10\/} (1959), 621--626.

\bibitem{weinstein-coho}
{\sc Weinstein, A.}
\newblock Cohomology of symplectomorphism groups and critical values of
  {H}amiltonians.
\newblock {\em Math. Z. 201}, 1 (1989), 75--82.

\end{thebibliography}
\end{document}